\documentclass[12pt,a4paper]{article}

\usepackage[T1]{fontenc}
\usepackage{lmodern}
\usepackage[english]{babel}

\usepackage[vlined]{algorithm2e}
\usepackage{amsmath}
\usepackage{amssymb}
\usepackage{amsthm}
\usepackage{booktabs}
\usepackage{hyperref}
\usepackage{microtype}
\usepackage{multicol}
\usepackage{subcaption}
\usepackage{tabularx}
\usepackage{tikz}

\usetikzlibrary{calc}

\newtheorem{theorem}{Theorem}
\newtheorem{corollary}{Corollary}

\makeatletter
\newcommand{\STS@text}[1]{{\rm STS}$(#1)$}
\newcommand{\STS@unstarred}[1]{\ifmmode{\text{\mbox{\STS@text{#1}}}}\else{\mbox{\STS@text{#1}}}\fi}
\newcommand{\STS@starred}[1]{\ifmmode{\text{\mbox{sub-\STS@text{#1}}}}\else{\mbox{sub-\STS@text{#1}}}\fi}
\newcommand{\STSS@unstarred}[1]{\ifmmode{\text{\mbox{\STS@text{#1}s}}}\else{\mbox{\STS@text{#1}s}}\fi}
\newcommand{\STSS@starred}[1]{\ifmmode{\text{\mbox{sub-\STS@text{#1}s}}}\else{\mbox{sub-\STS@text{#1}s}}\fi}
\newcommand{\STS}{\@ifstar{\STS@starred}{\STS@unstarred}}
\newcommand{\STSS}{\@ifstar{\STSS@starred}{\STSS@unstarred}}
\makeatother

\DontPrintSemicolon{}
\SetKw{KwCont}{continue}
\SetKwIF{If}{ElseIf}{Else}{if}{}{else if}{else}{endif}

\newcommand{\mypoint}[1]{\mathtt{#1}}

\begin{document}

\title{Algorithms and Complexity for Counting Configurations in Steiner Triple Systems}

\author{
Daniel Heinlein\thanks{Supported by the Academy of Finland, Grant 331044.}\phantom{ } and Patric R. J. \"Osterg\aa rd\\
Department of Communications and Networking\\
Aalto University School of Electrical Engineering\\
P.O.\ Box 15400, 00076 Aalto, Finland\\
\tt \{daniel.heinlein,patric.ostergard\}@aalto.fi
}

\date{}

\maketitle

\begin{abstract}
  Steiner triple systems form one of the most studied classes of
  combinatorial designs.
  Configurations, including subsystems, play a central role in the investigation
  of Steiner triple systems. With sporadic instances of small systems,
  ad-hoc algorithms for counting or listing configurations are typically fast enough
  for practical needs,
  but with many systems or large systems,
  the relevance of computational complexity and
  algorithms of low complexity is highlighted. General theoretical
  results as well as specific practical algorithms for important
  configurations are presented.
\end{abstract}

\noindent
    {\bf Keywords:} algorithm, computational complexity,
configuration, Steiner triple system, subsystem

\noindent
    {\bf MSC:} 05B07, 68Q25

\section{Introduction}

A \emph{Steiner triple system} (STS) is an ordered pair $(V,\mathcal{B})$,
where $V$ is a set of \emph{points} and $\mathcal{B}$ is a
set of \mbox{3-subsets} of points, called \emph{blocks} or \emph{lines}, such that
every \mbox{2-subset} of points occurs in exactly one block. The size of
the point set is the \emph{order} of the Steiner triple system, and a
Steiner triple system of order $v$ is denoted by \STS{v}.
It is well known that an \STS{v} exists iff
\begin{equation}\label{eq:sts}
v \equiv 1\text{ or }3\!\!\!\pmod{6}.
\end{equation}
An \STS{v} has $v(v-1)/6$ blocks and each point is in
$(v-1)/2$ blocks.
For more information about Steiner triple systems, see~\cite{C,CR}.

A \emph{configuration} in a $(V,\mathcal{B})$ \STS{v}
is a set system $(V',\mathcal{B'})$, where
$\mathcal{B'} \subseteq \mathcal{B}$ and $V' = \cup_{B \in \mathcal{B'}}B$.
For configurations, we
adopt the convention of calling the  elements of $\mathcal{B'}$ lines.
If each point in $V'$ occurs in at least two lines, then the configuration
is said to be \emph{full}.
A configuration that is an \STS{v} is called an \STS{v} \mbox{\emph{subsystem}}, or
a \emph{sub}-\STS{w}, and is
said to be proper if $w < v$ and nontrivial if $w > 3$.
A configuration with $w$ lines such that
each point is in three of the lines is a $w_3$ configuration~\cite{G}.
Double counting shows that the size of the point set of a $w_3$
configuration is $w$.

The computational problem of finding configurations
in designs is recurrent
in design theory. For example,
the problem of finding maximal arcs in projective planes of order 16,
studied in~\cite{GMT}, is about finding 2-(52,4,1) subdesigns in
2-(273,17,1) designs. Similar computational problems also occur in discrete
geometry~\cite{BP}.
We shall here focus explicitly on configurations in Steiner triple
systems, motivated by a need in~\cite{HO2} for algorithms to count
configurations in many Steiner triple systems with orders that are
large---even in the thousands. Earlier work in this area has mainly concerned
subsystems of Steiner triple systems~\cite{CCS}.

An \STS{v} is said to be \emph{isomorphic} to another \STS{v} if
there exists a bijection between the point sets that maps blocks
onto blocks; such a bijection is called an \emph{isomorphism}.
An isomorphism of a Steiner triple system onto itself
is an \emph{automorphism} of the Steiner triple system.
The automorphisms of a Steiner triple system form a group
under composition, the \emph{automorphism group} of the Steiner triple system.
These concepts are defined analogously for configurations.

The paper is organized as follows. In Section~\ref{sect:count},
the time complexity of counting and listing configurations in Steiner triple
systems is considered.
With fixed (sets of) configurations, these problems are in {\bf P}.
Polynomial upper bounds on the time complexity are obtained by developing
algorithms. The number of occurrences of all $n$-line configurations can be
obtained as a function of the number of occurrences of all full $m$-line
configurations with $m \leq n$.
The conjecture that a subset of the full $m$-line configurations does not
suffice has earlier been verified for $n \leq 7$, which is here extended to
$n \leq 8$.
Practical aspects are not addressed in the theoretical proofs, so
Section~\ref{sect:practice} is devoted to practical counting algorithms for
several specific small configurations.
In particular, an approach is developed for constructing an exhaustive
set of algorithms of a certain type. The algorithms
are compared experimentally, and the winning algorithms
are displayed for nine important configurations.

\section{Counting Configurations}\label{sect:count}

\subsection{Problem and Algorithms}

The computational problem studied here is as follows, where $\mathcal{C}$
is any \emph{fixed} set of configurations.

\vspace{5mm}
\noindent
\begin{minipage}{\textwidth}
{\bf Problem:} $P(\mathcal{C})$\\
{\bf Input:} A Steiner triple system $\mathcal{S}$ of order $v$\\
{\bf Output:} The number of occurrences of
the configurations in $\mathcal{S}$ that are isomorphic
to a configuration in~$\mathcal{C}$
\end{minipage}

\vspace{5mm}
The motivation for the work is that of \emph{counting}
configurations, but we will also address the problem of \emph{listing}
configurations in Steiner triple systems. The listed configurations can
obviously simultaneously be counted, but the fastest (known) counting algorithm
has in many cases smaller time complexity than the fastest (known) listing algorithm.
For example, considering a configuration with just one line,
an optimal algorithm simply lists all lines of
the \STS{v}, which takes $\Theta(v^2)$, whereas counting is simply a matter of
evaluating $v(v-1)/6$.

Indeed, for certain configurations the number of occurrences in an \STS{v}
only depends on $v$. Such configurations are called \emph{constant}. A configuration
that is not constant is said to be \emph{variable}. Small configurations in \STSS{v}
are surveyed in~\cite[Ch.~13]{CR}. All configurations with three
or fewer lines are constant and so are the members of five infinite families
presented in~\cite{HPWY}. A complete characterization of constant configurations
is still missing.

The fact that the set of configurations $\mathcal{C}$ is fixed gives
possibilities of simplifying proofs. For example, determining
whether two configurations are isomorphic can be done in constant time.
As the goal is to establish theoretical bounds, we do not make
any attempts to develop practical algorithms in this section but defer
such issues to Section~\ref{sect:practice}.

It is straightforward to see that Problem~$P(\mathcal{C})$ is in {\bf P}. Namely,
if there are at most $m$ lines in the configurations in $\mathcal{C}$, then
we can explore all subsets of at most $m$ lines of $\mathcal{S}$ and there
are $\Theta(v^{2m})$ such subsets. However, it will turn out that
this upper bound is weak, and better upper bounds---also in the context
of listing---will be obtained.
As $\mbox{\bf P} \subseteq \mbox{\bf PSPACE}$, the space complexity will also
be polynomial in all cases and will not be considered here.

The concept of configurations generated by sets
of points is central in the study of specific algorithms.
We here use a framework considered for Steiner
triple systems in~\cite[p.~99]{CR} and its references.
Given a configuration with point set $V$ and line set $\mathcal{B}$,
fix $V_0 \subseteq V$ and let
\[
V_{i+1} = V_i \cup \{z : \{x,y,z\} \in \mathcal{B},\ x,y \in V_i\}.
\]
For some finite $j$, $V_{j+1} = V_j$ and then $V_i = V_j$ for all
$i > j$. The set $V_j$ is the \emph{closure} of $V_0$, and
$V_j$ is the point set of a subconfiguration, which is the smallest
point-induced subconfiguration (subsystem, when $(V,\mathcal{B})$ is
a Steiner triple system)
containing $V_0$ and is said to be \emph{generated} by $V_0$.

If $|V_0| = 2$, then $V_0$ generates at most one line, a line in which the
pair of points occurs. Similarly, if $|V_0|=3$ and
$V_0$ is a line, then no further lines are generated.
But many types of configurations can be generated
by $V_0$ when $|V_0|=3$ and $V_0$ is not a line.

\begin{theorem}\label{thm:general}
Let $\mathcal{C}$ be a collection of configurations that
can be generated by $m$ points. Then the time complexity of listing
the configurations in an\/ \STS{v} isomorphic to a configuration in
$\mathcal{C}$  is $O(v^m)$.
\end{theorem}

\begin{proof}
Consider a $(V,\mathcal{B})$ \STS{v}. For each possible $m$-subset
$V'_0 \subseteq V$ and with $\mathcal{B}'_0 = \emptyset$,
the following iterative extension procedure is carried out in all possible ways:
given a point set $V'_i$ and a line set $\mathcal{B}'_i$, let
$V'_{i+1} = V'_{i} \cup B$ and $\mathcal{B}'_{i+1} = \mathcal{B}'_{i} \cup \{B\}$,
where $B \in \mathcal{B}\setminus \mathcal{B}'_i$ and $|B \cap V'_{i}|\geq 2$.
Whenever $i$ equals the number of lines of
a configuration in $\mathcal{C}$, an isomorphism test is carried out.
If the outcome of that test is positive, then the configuration is listed if
two additional tests are passed: (i) the $m$-subset $V'_0$ is the lexicographically
smallest one amongst the $m$-subsets from which the configuration can
be generated, and (ii) the configuration has not already been listed in the
branch of the search tree starting from $V'_0$. This makes sure that each
configuration is listed
exactly once. The number of lines in the configurations in $\mathcal{C}$
sets a bound on the largest value of $i$ to consider.

Using a precomputed data structure, which can be created in $O(v^2)$
(cf. Section~\ref{sect:practice}), the extension can be carried out in
constant time.
As the isomorphism test can also be carried out in constant time, the
time complexity of the problem is bounded from above by the number of
$m$-subsets of a $v$-set and is therefore $O(v^m)$.
\end{proof}

Note that the core of the algorithm in the proof is essentially about
canonical augmentation~\cite{M0}---see also~\cite[Sect.\ 4.2.3]{KO}---which
consists of (i) a parent test and (ii) an isomorphism test.
The extension procedures in and before the proof of Theorem~\ref{thm:general} are
closely connected to the core of Miller's algorithm~\cite{M} for
computing a canonical form of an \STS{v} in $O(v^{\log v + O(1)})$. Further
related studies include~\cite{CCS,S}.
Also note the similarity between the extension procedure and the algorithm
in~\cite{BP}.

For listing algorithms, it is now a matter of determining the size of
point sets needed to generate configurations.

\subsection{Small Configurations}

An algorithm that solves Problem~$P(\mathcal{C})$ gives
an upper bound on the time complexity. For some configurations,
including the smallest nontrivial case of \emph{Pasch} configurations,
it is possible to prove that the upper bound given by Theorem~\ref{thm:general}
is actually exact. The Pasch configuration is depicted
in Figure~\ref{sf:pasch}; one possible set of generating
points is here and in later pictures shown with bold circles.
The labels in all pictures refer to the naming of variables in
Section~\ref{sect:practice}.

\begin{figure}[htbp]
\vspace{5mm}
\centering
\begin{subfigure}[b]{0.4\textwidth}
\centering
\begin{tikzpicture}
\coordinate (d) at (0:0);
\coordinate (f) at (30:1);
\coordinate (a) at (90:2);
\coordinate (b) at (150:1);
\coordinate (e) at (210:2);
\coordinate (c) at (330:2);

\draw (a)--(b)--(e);
\draw (a)--(f)--(c);
\draw (b)--(d)--(c);
\draw (e)--(d)--(f);

\foreach \x in {a,b,f} \node[draw,fill=white,ultra thick,circle,text width=width("b"),text centered] at (\x) {$\mypoint{\x}$};
\foreach \x in {c,d,e} \node[draw,fill=white,dotted,circle,text width=width("b"),text centered] at (\x) {$\mypoint{\x}$};
\end{tikzpicture}
\caption{Pasch}\label{sf:pasch}
\end{subfigure}
\begin{subfigure}[b]{0.4\textwidth}
\centering
\begin{tikzpicture}
\coordinate (e) at (30:1);
\coordinate (g) at (90:2);
\coordinate (b) at (150:1);
\coordinate (a) at (210:2);
\coordinate (f) at (330:2);
\coordinate (d) at ($(b)!0.5!(e)$);
\coordinate (c) at ($(a)!0.5!(f)$);

\draw (a)--(b)--(g);
\draw (a)--(c)--(f);
\draw (b)--(d)--(e);
\draw (c)--(d)--(g);
\draw (f)--(e)--(g);

\foreach \x in {a,c,e} \node[draw,fill=white,ultra thick,circle,text width=width("b"),text centered] at (\x) {$\mypoint{\x}$};
\foreach \x in {b,d,f,g} \node[draw,fill=white,dotted,circle,text width=width("b"),text centered] at (\x) {$\mypoint{\x}$};
\end{tikzpicture}
\caption{Mitre}\label{sf:mitre}
\end{subfigure}
\caption{The Pasch and mitre configurations}\label{fig:pasch}
\end{figure}
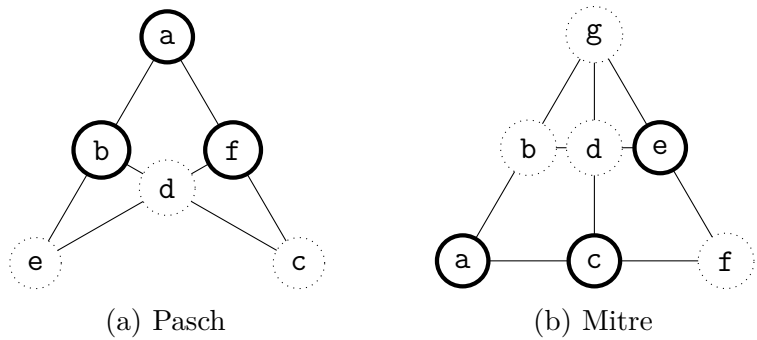

\begin{theorem}\label{thm:pasch}
The time complexity of listing the Pasch configurations in an\/
\STS{v} is $\Theta(v^3)$.
\end{theorem}

\begin{proof}
The Pasch configuration can be generated by the 3 points
indicated in Figure~\ref{sf:pasch}. Then
combine the upper bound given by Theorem~\ref{thm:general} and the fact that
there are \STSS{v} with $v(v-1)(v-3)/24$ Pasch configurations~\cite{SW}.
\end{proof}

\begin{theorem}\label{thm:4}
The time complexity of counting the number of occurrences of a
given\/ $4$-line variable configuration in an\/ \STS{v} is $O(v^3)$.
\end{theorem}

\begin{proof}
For any of the
11 variable $4$-line configurations, the number of occurrences can be derived
from the number of occurrences of any single one of them~\cite{GGM}. In particular,
using the number of Pasch configurations, the result follows from
Theorem~\ref{thm:pasch}.
\end{proof}

The case of 4-line configurations provides
further examples with different time complexity
for counting and listing. Namely, for all but the Pasch configuration,
4 to 8 points are needed for generation, which gives listing algorithms with
time complexity $O(v^4)$ to $O(v^8)$. In fact, the maximum number of occurrences of
configurations~\cite{GGM} shows that the time complexity of listing
is $\Theta(v^4)$ to $\Theta(v^8)$. For example, the configuration that
requires 8 points for generation consists of four disjoint lines.

The result in Theorem~\ref{thm:4} may be extended by considering $n$-line
configurations for any fixed $n>4$ in an analogous way. Formulas for
the relationship between the numbers of occurrences of variable $5$-line
and $6$-line configurations can be found in~\cite{DMGG} and~\cite{F},
respectively. Actually, since we do not need exact formulas in our
study of complexity, the following general result will suffice.

\begin{theorem}[\cite{HPWY}]\label{thm:number}
The number of occurrences of any variable $n$-line configuration in
an\/ \STS{v} is a constant (as a polynomial in $v$) plus a linear combination
(with coefficients that are polynomials in $v$) of the numbers of occurrences of
the full $m$-line configurations with $m \leq n$.
\end{theorem}

The only full $5$-line configuration is
the \emph{mitre} configuration (Figure~\ref{sf:mitre}).

\begin{theorem}\label{thm:5}
The time complexity of counting the number of occurrences of a
given\/ $5$-line variable configuration in an\/ \STS{v} is $O(v^3)$.
\end{theorem}

\begin{proof}
The full $m$-line configurations with $m \leq 5$ are the Pasch and
the mitre configuration.
The mitre configuration can be generated by 3 points---as indicated in
Figure~\ref{sf:mitre}---and
therefore the number of occurrences can be obtained in $O(v^3)$
by Theorem~\ref{thm:general}.
As also the Pasch configurations can be counted in $O(v^3)$ by Theorem~\ref{thm:4},
the result now follows from Theorem~\ref{thm:number}.
\end{proof}

With an increasing number of lines, extending the results in
Theorems~\ref{thm:4} and~\ref{thm:5} is rather straightforward but
becomes more and more laborious. Moreover, for different configurations
we will get different upper bounds on the time complexity,
so the results cannot be stated
in compact form. For example, for $6$-line and $7$-line configurations
we get the following general result.

\begin{theorem}\label{thm:6}
The time complexity of counting the number of occurrences of a
given\/ $6$-line\/ or $7$-line variable configuration in an\/ \STS{v} is $O(v^4)$.
\end{theorem}

\begin{proof}
There are 5 full $6$-line configurations, which
are depicted in Figure~\ref{fig:full} with generating sets indicated. Consequently,
by Theorem~\ref{thm:number}, the time complexity for counting the number of
occurrences of a variable $6$-line configuration is $O(v^4)$.
For variable $7$-line configurations, a similar argument applies as all full
$7$-line configurations have generating sets of size at most 4 (by
Table~\ref{tab:full}, to be discussed later).
\end{proof}

Clearly Theorem~\ref{thm:6} is not tight in the sense that for some of the variable
$6$-line and $7$-line configurations, the time complexity
of counting is $O(v^3)$.

\begin{figure}[htbp]
\centering
\begin{subfigure}[b]{0.3\textwidth}
\centering
\begin{tikzpicture}
\coordinate (f) at (0:0);
\coordinate (d) at (30:1);
\coordinate (g) at (90:2);
\coordinate (a) at (150:1);
\coordinate (c) at (210:2);
\coordinate (b) at (270:1);
\coordinate (e) at (330:2);

\draw (c)--(a)--(g);
\draw (a)--(f)--(e);
\draw (c)--(b)--(e);
\draw (b)--(f)--(g);
\draw (c)--(f)--(d);
\draw (e)--(d)--(g);

\foreach \x in {c,e,f} \node[draw,fill=white,ultra thick,circle,text width=width("b"),text centered] at (\x) {$\mypoint{\x}$};
\foreach \x in {a,b,d,g} \node[draw,fill=white,dotted,circle,text width=width("b"),text centered] at (\x) {$\mypoint{\x}$};
\end{tikzpicture}
\caption{Fano--line}\label{sf:fanoline}
\end{subfigure}
\begin{subfigure}[b]{0.3\textwidth}
\centering
\begin{tikzpicture}
\coordinate (a) at (1,3);
\coordinate (b) at (2,3);
\coordinate (d) at (3,3);
\coordinate (e) at (1,2);
\coordinate (c) at (2,2);
\coordinate (h) at (3,2);
\coordinate (g) at (1,1);
\coordinate (f) at (2,1);
\coordinate (i) at (3,1);

\draw (a)--(b)--(d);
\draw (a)--(e)--(g);
\draw (b)--(c)--(f);
\draw (e)--(c)--(h);
\draw (d)--(h)--(i);
\draw (g)--(f)--(i);

\foreach \x in {a,b,c,e} \node[draw,fill=white,ultra thick,circle,text width=width("b"),text centered] at (\x) {$\mypoint{\x}$};
\foreach \x in {d,f,g,h,i} \node[draw,fill=white,dotted,circle,text width=width("b"),text centered] at (\x) {$\mypoint{\x}$};
\end{tikzpicture}
\caption{Grid}\label{sf:grid}
\end{subfigure}
\begin{subfigure}[b]{0.3\textwidth}
\centering
\begin{tikzpicture}
\coordinate (f) at (3,3);
\coordinate (h) at (3,2);
\coordinate (e) at (1,1);
\coordinate (d) at (5,1);
\coordinate (a) at (3,0);
\coordinate (b) at ($(e)!0.35!(a)$);
\coordinate (g) at ($(d)!0.35!(a)$);
\coordinate (c) at (intersection of e--h and b--f);
\coordinate (i) at (intersection of h--d and f--g);

\draw (a)--(b)--(e);
\draw (a)--(g)--(d);
\draw (b)--(c)--(f);
\draw (e)--(c)--(h);
\draw (d)--(i)--(h);
\draw (f)--(i)--(g);

\foreach \x in {a,b,d,f} \node[draw,fill=white,ultra thick,circle,text width=width("b"),text centered] at (\x) {$\mypoint{\x}$};
\foreach \x in {c,e,g,h,i} \node[draw,fill=white,dotted,circle,text width=width("b"),text centered] at (\x) {$\mypoint{\x}$};
\end{tikzpicture}
\caption{Prism}\label{sf:prism}
\end{subfigure}
\begin{subfigure}[b]{0.4\textwidth}
\centering
\begin{tikzpicture}
\coordinate (b) at (1,3);
\coordinate (a) at (3,3);
\coordinate (c) at (5,3);
\coordinate (d) at (2,2);
\coordinate (g) at (4,2);
\coordinate (e) at (1,1);
\coordinate (h) at (3,1);
\coordinate (f) at (5,1);

\draw (b)--(a)--(c);
\draw (e)--(d)--(a);
\draw (a)--(g)--(f);
\draw (b)--(d)--(h);
\draw (c)--(g)--(h);
\draw (e)--(h)--(f);

\foreach \x in {b,c,d} \node[draw,fill=white,ultra thick,circle,text width=width("b"),text centered] at (\x) {$\mypoint{\x}$};
\foreach \x in {a,e,f,g,h} \node[draw,fill=white,dotted,circle,text width=width("b"),text centered] at (\x) {$\mypoint{\x}$};
\end{tikzpicture}
\caption{Hexagon}\label{sf:hexagon}
\end{subfigure}
\begin{subfigure}[b]{0.4\textwidth}
\centering
\begin{tikzpicture}
\coordinate (g) at (3,4);
\coordinate (a) at (1,1);
\coordinate (c) at (5,1);
\coordinate (b) at ($(a)!0.7!(c)$);
\coordinate (h) at ($(c)!0.5!(g)$);
\coordinate (f) at ($(a)!0.5!(g)$);
\coordinate (d) at ($(b)!0.4!(f)$);
\coordinate (e) at (intersection of f--h and d--g);

\draw (a)--(b)--(c);
\draw (a)--(f)--(g);
\draw (b)--(d)--(f);
\draw (c)--(h)--(g);
\draw (d)--(e)--(g);
\draw (f)--(e)--(h);

\foreach \x in {e,f,g} \node[draw,fill=white,ultra thick,circle,text width=width("b"),text centered] at (\x) {$\mypoint{\x}$};
\foreach \x in {a,b,c,d,h} \node[draw,fill=white,dotted,circle,text width=width("b"),text centered] at (\x) {$\mypoint{\x}$};
\end{tikzpicture}
\caption{Crown}\label{sf:crown}
\end{subfigure}
\caption{The full $6$-line configurations}\label{fig:full}
\end{figure}
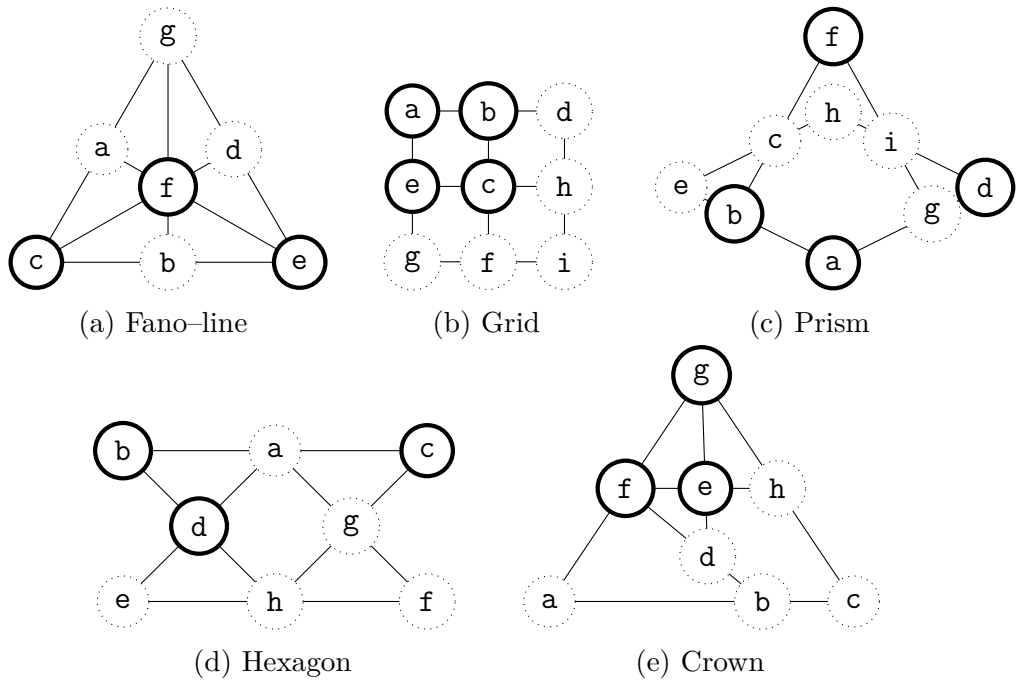

For a fixed $n \geq 7$, there are too many full $n$-line configurations
to depict all of them here. However, computationally one can easily get rather
extensive results. The number $N$ of isomorphism classes of
full $n$-line configurations for $n \leq 12$ have been obtained earlier
in~\cite{FGG}. We extend that work to $n \leq 13$ in Table~\ref{tab:full}
and, for all those parameters, tabulate the distribution $N_i$
based on the size $i$ of the smallest generating set.
Also, the distribution of automorphism group sizes, $|\operatorname{Aut}|$, is shown.

\begin{table}[htbp]
\caption{Sizes of generating sets of full $n$-line configurations}\label{tab:full}
\vspace*{-5mm}
\begin{center}
\footnotesize
\setlength{\tabcolsep}{3pt}
\begin{tabularx}{\textwidth}{rrXrrrrrrr}
\toprule
$n$ & $N$ & $|\operatorname{Aut}|$ & $N_3$ & $N_4$ & $N_5$ & $N_6$ & $N_7$ & $N_8$ & $N_9$ \\
\midrule
4 & 1 & $ { 24 }^{ 1 } \allowbreak $ & 1 & 0 & 0 & 0 & 0 & 0 & 0 \\
5 & 1 & $ { 12 }^{ 1 } \allowbreak $ & 1 & 0 & 0 & 0 & 0 & 0 & 0 \\
6 & 5 & $ { 2 }^{ 1 } \allowbreak { 12 }^{ 2 } \allowbreak { 24 }^{ 1 } \allowbreak { 72 }^{ 1 } \allowbreak $ & 3 & 2 & 0 & 0 & 0 & 0 & 0 \\
7 & 19 & $ { 1 }^{ 3 } \allowbreak { 2 }^{ 5 } \allowbreak { 4 }^{ 3 } \allowbreak { 6 }^{ 5 } \allowbreak { 12 }^{ 2 } \allowbreak { 168 }^{ 1 } \allowbreak $ & 13 & 6 & 0 & 0 & 0 & 0 & 0 \\
8 & 153 & $ { 1 }^{ 58 } \allowbreak { 2 }^{ 50 } \allowbreak { 3 }^{ 1 } \allowbreak { 4 }^{ 22 } \allowbreak { 6 }^{ 2 } \allowbreak { 8 }^{ 6 } \allowbreak { 12 }^{ 3 } \allowbreak { 16 }^{ 5 } \allowbreak { 24 }^{ 1 } \allowbreak { 32 }^{ 1 } \allowbreak { 48 }^{ 2 } \allowbreak { 64 }^{ 1 } \allowbreak { 1152 }^{ 1 } \allowbreak $ & 98 & 48 & 6 & 1 & 0 & 0 & 0 \\
9 & 1615 & $ { 1 }^{ 1156 } \allowbreak { 2 }^{ 341 } \allowbreak { 3 }^{ 5 } \allowbreak { 4 }^{ 55 } \allowbreak { 6 }^{ 15 } \allowbreak { 8 }^{ 19 } \allowbreak { 9 }^{ 1 } \allowbreak { 12 }^{ 10 } \allowbreak { 16 }^{ 4 } \allowbreak { 18 }^{ 1 } \allowbreak { 24 }^{ 2 } \allowbreak { 32 }^{ 2 } \allowbreak { 36 }^{ 1 } \allowbreak { 48 }^{ 1 } \allowbreak { 108 }^{ 1 } \allowbreak { 288 }^{ 1 } \allowbreak $ & 1081 & 492 & 41 & 1 & 0 & 0 & 0 \\
10 & 25180 & $ { 1 }^{ 21970 } \allowbreak { 2 }^{ 2533 } \allowbreak { 3 }^{ 24 } \allowbreak { 4 }^{ 421 } \allowbreak { 5 }^{ 1 } \allowbreak { 6 }^{ 44 } \allowbreak { 8 }^{ 102 } \allowbreak { 10 }^{ 2 } \allowbreak { 12 }^{ 16 } \allowbreak { 16 }^{ 25 } \allowbreak { 18 }^{ 1 } \allowbreak { 20 }^{ 4 } \allowbreak { 24 }^{ 11 } \allowbreak { 32 }^{ 9 } \allowbreak { 36 }^{ 2 } \allowbreak { 48 }^{ 4 } \allowbreak { 64 }^{ 1 } \allowbreak { 96 }^{ 1 } \allowbreak { 120 }^{ 2 } \allowbreak { 128 }^{ 1 } \allowbreak { 288 }^{ 4 } \allowbreak { 576 }^{ 1 } \allowbreak { 1728 }^{ 1 } \allowbreak $ & 17038 & 7426 & 688 & 26 & 2 & 0 & 0 \\
11 & 479238 & $ { 1 }^{ 454542 } \allowbreak { 2 }^{ 21449 } \allowbreak { 3 }^{ 41 } \allowbreak { 4 }^{ 2494 } \allowbreak { 6 }^{ 67 } \allowbreak { 8 }^{ 456 } \allowbreak { 11 }^{ 1 } \allowbreak { 12 }^{ 40 } \allowbreak { 16 }^{ 72 } \allowbreak { 24 }^{ 27 } \allowbreak { 32 }^{ 14 } \allowbreak { 36 }^{ 1 } \allowbreak { 48 }^{ 12 } \allowbreak { 64 }^{ 2 } \allowbreak { 72 }^{ 2 } \allowbreak { 96 }^{ 6 } \allowbreak { 144 }^{ 7 } \allowbreak { 288 }^{ 3 } \allowbreak { 864 }^{ 1 } \allowbreak { 4032 }^{ 1 } \allowbreak $ & 323591 & 142075 & 13193 & 371 & 8 & 0 & 0 \\
12 & 10695820 & $ { 1 }^{ 10431210 } \allowbreak { 2 }^{ 236254 } \allowbreak { 3 }^{ 229 } \allowbreak { 4 }^{ 23407 } \allowbreak { 6 }^{ 267 } \allowbreak { 8 }^{ 3234 } \allowbreak { 12 }^{ 224 } \allowbreak { 16 }^{ 510 } \allowbreak { 18 }^{ 3 } \allowbreak { 24 }^{ 164 } \allowbreak { 32 }^{ 100 } \allowbreak { 36 }^{ 7 } \allowbreak { 48 }^{ 97 } \allowbreak { 64 }^{ 20 } \allowbreak { 72 }^{ 11 } \allowbreak { 96 }^{ 30 } \allowbreak { 128 }^{ 7 } \allowbreak { 144 }^{ 9 } \allowbreak { 192 }^{ 8 } \allowbreak { 288 }^{ 8 } \allowbreak { 384 }^{ 7 } \allowbreak { 432 }^{ 1 } \allowbreak { 576 }^{ 1 } \allowbreak { 768 }^{ 2 } \allowbreak { 864 }^{ 2 } \allowbreak { 1152 }^{ 3 } \allowbreak { 1536 }^{ 1 } \allowbreak { 1728 }^{ 1 } \allowbreak { 2016 }^{ 1 } \allowbreak { 10368 }^{ 1 } \allowbreak { 82944 }^{ 1 } \allowbreak $ & 7087335 & 3289199 & 308659 & 10447 & 170 & 9 & 1 \\
13 & 270939475 & $ { 1 }^{ 267284483 } \allowbreak { 2 }^{ 3338204 } \allowbreak { 3 }^{ 1034 } \allowbreak { 4 }^{ 285033 } \allowbreak { 6 }^{ 1388 } \allowbreak { 8 }^{ 22125 } \allowbreak { 12 }^{ 1748 } \allowbreak { 13 }^{ 1 } \allowbreak { 16 }^{ 2502 } \allowbreak { 18 }^{ 9 } \allowbreak { 24 }^{ 1713 } \allowbreak { 32 }^{ 403 } \allowbreak { 36 }^{ 15 } \allowbreak { 39 }^{ 1 } \allowbreak { 48 }^{ 501 } \allowbreak { 64 }^{ 68 } \allowbreak { 72 }^{ 28 } \allowbreak { 96 }^{ 90 } \allowbreak { 128 }^{ 10 } \allowbreak { 144 }^{ 35 } \allowbreak { 192 }^{ 30 } \allowbreak { 216 }^{ 1 } \allowbreak { 256 }^{ 3 } \allowbreak { 288 }^{ 16 } \allowbreak { 336 }^{ 1 } \allowbreak { 384 }^{ 10 } \allowbreak { 432 }^{ 6 } \allowbreak { 576 }^{ 4 } \allowbreak { 768 }^{ 3 } \allowbreak { 864 }^{ 3 } \allowbreak { 1152 }^{ 1 } \allowbreak { 2016 }^{ 2 } \allowbreak { 2592 }^{ 1 } \allowbreak { 4032 }^{ 1 } \allowbreak { 12096 }^{ 1 } \allowbreak { 13824 }^{ 1 } \allowbreak $ & 175420488 & 87098667 & 8100133 & 315860 & 4266 & 60 & 1 \\
\bottomrule
\end{tabularx}
\end{center}
\end{table}

It is an interesting open question whether all full configurations are
really required in Theorem~\ref{thm:number} or whether a subset of them
would suffice. It is conjectured in~\cite{HPWY} that Theorem~\ref{thm:number}
is indeed strict. For $n = 4$, 5, and 6, this follows from the formulas
of~\cite{GGM},~\cite{DMGG}, and~\cite{F}, respectively, and the case of $n=6$
is handled explicitly in~\cite{HPWY}. The conjecture has also been verified
for $n=7$ in an unpublished study~\cite{U}.

An established approach~\cite{GG} to study the aforementioned conjecture
for small values of $n$
is to investigate the number of occurrences of full $m$-line configurations for
$m \le n$ in a set of (randomly chosen)
\STSS{v} for some fixed $v$. If there are $r$ configurations to consider,
we get for each \STS{v} a vector of length $r+1$
with nonnegative integers ($r$ counts and a constant, say 1).
Forming a matrix with rows consisting of those vectors, we check whether the
rank is $r+1$ (for which we obviously need at least $r+1$ vectors).

In the current study, this approach is used to
extend the earlier results to $n = 8$. Notice that the
following theorem actually confirms the old results
for $n \leq 7$, including the unpublished result~\cite{U}
for $n=7$.

\begin{theorem}\label{thm:8}
For $n \leq 8$, there is no full $n$-line configuration $C$ whose number
of occurrences in an\/ \STS{v} is a constant (as a function of $v$)
plus a linear combination (with coefficients that are functions
of $v$) of the numbers of occurrences of the full $m$-line
configurations with $m \leq n$ excluding $C$.
\end{theorem}

\begin{proof}
For $n=8$, 623 distinct and randomly chosen \STSS{25} were considered,
and for each of these a vector of length 180 was determined as
described earlier (there are 179 full $m$-line configurations with
$m \leq 8$). Calculations using {\tt gap} show that the $623\times 180$
matrix formed by these vectors has rank 180.

The 623 Steiner triple systems considered
and the related vectors are published separately in~\cite{HO0}.
\end{proof}

See~\cite{B,C0,GG} for
further results on configurations in designs in general and in
Steiner triple systems in particular.

\subsection{\texorpdfstring{$w_3$ Configurations}{w3 Configurations}}

For configurations with a large number of lines, a more general
study is feasible only for specific types of configurations. We
here consider $w_3$ configurations, algorithms for which are needed
in~\cite{HO2}.
The number $N$ of isomorphism classes of $w_3$ configurations with small
$w$ can be found in~\cite{BBP,G}.
In Table~\ref{tab:w3} we list those values for $7 \leq n \leq 16$,
and for each entry we
give the same information as in Table~\ref{tab:full}.

\begin{table}[htbp]
\vspace{5mm}
\caption{Sizes of generating sets of $w_3$ configurations}\label{tab:w3}
\vspace*{-5mm}
\begin{center}
\small
\begin{tabularx}{\textwidth}{rrXrrrr}
\toprule
$w$ & $N$ & $|\operatorname{Aut}|$ & $N_3$ & $N_4$ & $N_5$ & $N_6$ \\
\midrule
7 & 1 & $ { 168 }^{ 1 } \allowbreak $ & 1 & 0 & 0 & 0 \\
8 & 1 & $ { 48 }^{ 1 } \allowbreak $ & 1 & 0 & 0 & 0 \\
9 & 3 & $ { 9 }^{ 1 } \allowbreak { 12 }^{ 1 } \allowbreak { 108 }^{ 1 } \allowbreak $ & 3 & 0 & 0 & 0 \\
10 & 10 & $ { 2 }^{ 1 } \allowbreak { 3 }^{ 2 } \allowbreak { 4 }^{ 2 } \allowbreak { 6 }^{ 1 } \allowbreak { 10 }^{ 1 } \allowbreak { 12 }^{ 1 } \allowbreak { 24 }^{ 1 } \allowbreak { 120 }^{ 1 } \allowbreak $ & 9 & 1 & 0 & 0 \\
11 & 31 & $ { 1 }^{ 10 } \allowbreak { 2 }^{ 13 } \allowbreak { 3 }^{ 1 } \allowbreak { 4 }^{ 2 } \allowbreak { 6 }^{ 3 } \allowbreak { 8 }^{ 1 } \allowbreak { 11 }^{ 1 } \allowbreak $ & 31 & 0 & 0 & 0 \\
12 & 229 & $ { 1 }^{ 146 } \allowbreak { 2 }^{ 60 } \allowbreak { 3 }^{ 3 } \allowbreak { 4 }^{ 3 } \allowbreak { 6 }^{ 8 } \allowbreak { 8 }^{ 1 } \allowbreak { 12 }^{ 3 } \allowbreak { 18 }^{ 1 } \allowbreak { 24 }^{ 1 } \allowbreak { 32 }^{ 1 } \allowbreak { 36 }^{ 1 } \allowbreak { 72 }^{ 1 } \allowbreak $ & 224 & 5 & 0 & 0 \\
13 & 2036 & $ { 1 }^{ 1770 } \allowbreak { 2 }^{ 190 } \allowbreak { 3 }^{ 20 } \allowbreak { 4 }^{ 30 } \allowbreak { 6 }^{ 16 } \allowbreak { 8 }^{ 3 } \allowbreak { 12 }^{ 4 } \allowbreak { 13 }^{ 1 } \allowbreak { 39 }^{ 1 } \allowbreak { 96 }^{ 1 } \allowbreak $ & 2010 & 26 & 0 & 0 \\
14 & 21399 & $ { 1 }^{ 20328 } \allowbreak { 2 }^{ 916 } \allowbreak { 3 }^{ 19 } \allowbreak { 4 }^{ 91 } \allowbreak { 6 }^{ 12 } \allowbreak { 7 }^{ 1 } \allowbreak { 8 }^{ 15 } \allowbreak { 12 }^{ 7 } \allowbreak { 14 }^{ 3 } \allowbreak { 16 }^{ 3 } \allowbreak { 24 }^{ 2 } \allowbreak { 128 }^{ 1 } \allowbreak { 56448 }^{ 1 } \allowbreak $ & 20798 & 599 & 1 & 1 \\
15 & 245342 & $ { 1 }^{ 241240 } \allowbreak { 2 }^{ 3709 } \allowbreak { 3 }^{ 69 } \allowbreak { 4 }^{ 180 } \allowbreak { 5 }^{ 5 } \allowbreak { 6 }^{ 59 } \allowbreak { 8 }^{ 34 } \allowbreak { 10 }^{ 3 } \allowbreak { 12 }^{ 11 } \allowbreak { 15 }^{ 2 } \allowbreak { 16 }^{ 10 } \allowbreak { 18 }^{ 1 } \allowbreak { 20 }^{ 2 } \allowbreak { 24 }^{ 2 } \allowbreak { 30 }^{ 2 } \allowbreak { 32 }^{ 1 } \allowbreak { 48 }^{ 6 } \allowbreak { 72 }^{ 1 } \allowbreak { 128 }^{ 1 } \allowbreak { 192 }^{ 2 } \allowbreak { 720 }^{ 1 } \allowbreak { 8064 }^{ 1 } \allowbreak $ & 222524 & 22809 & 8 & 1 \\
16 & 3004881 & $ { 1 }^{ 2986560 } \allowbreak { 2 }^{ 17119 } \allowbreak { 3 }^{ 320 } \allowbreak { 4 }^{ 635 } \allowbreak { 6 }^{ 88 } \allowbreak { 8 }^{ 93 } \allowbreak { 12 }^{ 19 } \allowbreak { 16 }^{ 24 } \allowbreak { 24 }^{ 7 } \allowbreak { 32 }^{ 5 } \allowbreak { 48 }^{ 5 } \allowbreak { 96 }^{ 2 } \allowbreak { 1512 }^{ 1 } \allowbreak { 2016 }^{ 1 } \allowbreak { 4608 }^{ 1 } \allowbreak { 18144 }^{ 1 } \allowbreak $ & 2260797 & 744045 & 35 & 4 \\
\bottomrule
\end{tabularx}
\end{center}
\end{table}

The unique $7_3$ and $8_3$ configurations are the Fano plane and
the M\"obius--Kantor configuration, respectively, and are depicted
in Figure~\ref{fig:fano}, again with generating sets indicated.

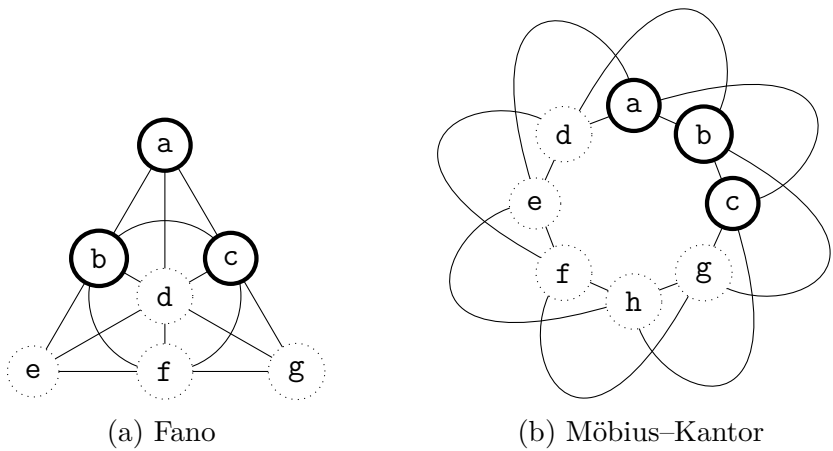
\begin{figure}[htbp]
\centering
\begin{subfigure}[b]{0.3\textwidth}
\centering
\begin{tikzpicture}
\coordinate (d) at (0:0);
\coordinate (c) at (30:1);
\coordinate (a) at (90:2);
\coordinate (b) at (150:1);
\coordinate (e) at (210:2);
\coordinate (f) at (270:1);
\coordinate (g) at (330:2);

\draw (a)--(b)--(e);
\draw (a)--(c)--(g);
\draw (a)--(d)--(f);
\draw (b)--(d)--(g);
\draw (c)--(d)--(e);
\draw (e)--(f)--(g);
\draw (d) circle [radius=1];

\foreach \x in {a,b,c} \node[draw,fill=white,ultra thick,circle,text width=width("b"),text centered] at (\x) {$\mypoint{\x}$};
\foreach \x in {d,e,f,g} \node[draw,fill=white,dotted,circle,text width=width("b"),text centered] at (\x) {$\mypoint{\x}$};
\end{tikzpicture}
\caption{Fano}\label{sf:fano}
\end{subfigure}
\begin{subfigure}[b]{0.6\textwidth}
\centering
\hspace*{-0.7cm}
\begin{tikzpicture}[scale=0.65]
\coordinate (a) at (90:2);
\coordinate (b) at (45:2);
\coordinate (c) at (0:2);
\coordinate (g) at (315:2);
\coordinate (h) at (270:2);
\coordinate (f) at (225:2);
\coordinate (e) at (180:2);
\coordinate (d) at (135:2);

\coordinate (ab) at ($(a)!5!(b)$);
\coordinate (bc) at ($(b)!5!(c)$);
\coordinate (cg) at ($(c)!5!(g)$);
\coordinate (gh) at ($(g)!5!(h)$);
\coordinate (hf) at ($(h)!5!(f)$);
\coordinate (fe) at ($(f)!5!(e)$);
\coordinate (ed) at ($(e)!5!(d)$);
\coordinate (da) at ($(d)!5!(a)$);

\coordinate (aa) at ($(0,0)!2!(a)$);
\coordinate (bb) at ($(0,0)!2!(b)$);
\coordinate (cc) at ($(0,0)!2!(c)$);
\coordinate (gg) at ($(0,0)!2!(g)$);
\coordinate (hh) at ($(0,0)!2!(h)$);
\coordinate (ff) at ($(0,0)!2!(f)$);
\coordinate (ee) at ($(0,0)!2!(e)$);
\coordinate (dd) at ($(0,0)!2!(d)$);

\draw (a) .. controls (ab) and (gg) .. (g);
\draw (b) .. controls (bc) and (hh) .. (h);
\draw (c) .. controls (cg) and (ff) .. (f);
\draw (g) .. controls (gh) and (ee) .. (e);
\draw (h) .. controls (hf) and (dd) .. (d);
\draw (f) .. controls (fe) and (aa) .. (a);
\draw (e) .. controls (ed) and (bb) .. (b);
\draw (d) .. controls (da) and (cc) .. (c);

\foreach \x in {a,b,c} \node[draw,fill=white,ultra thick,circle,text width=width("b"),text centered] at (\x) {$\mypoint{\x}$};
\foreach \x in {d,e,f,g,h} \node[draw,fill=white,dotted,circle,text width=width("b"),text centered] at (\x) {$\mypoint{\x}$};
\end{tikzpicture}
\vskip -2cm
\caption{M\"{o}bius--Kantor}\label{sf:mobiuskantor}
\end{subfigure}
\caption{The $7_3$ and $8_3$ configurations}\label{fig:fano}
\end{figure}

We call an $n_3$ configuration in a $v_3$ configuration a
\emph{subconfiguration} and say that such a subconfiguration
is \emph{proper} if $n < v$.

\begin{theorem}
If a $v_3$ configuration has a proper $n_3$ subconfiguration, then
it has a proper $(n-v)_3$ subconfiguration.
\end{theorem}

\begin{proof}
Consider the configuration obtained by removing the
points and lines of a proper $n_3$ subconfiguration from
the $v_3$ configuration.
\end{proof}

\begin{corollary}
If a $v_3$ configuration has a proper $n_3$ subconfiguration, then
$v \geq 14$.
\end{corollary}

The smallest example of a $v_3$ configuration
with proper $n_3$ subconfigurations is of type $14_3$ and is
unique as the $7_3$ configuration is unique.
This particular configuration occurs in the \STSS{21} of Wilson
type, as discussed in~\cite{HO1}.

The size of the generating set of a $v_3$ configuration with proper
$n_3$ and $(v-n)_3$ subconfigurations equals the sum of the sizes of the
generating sets of those subconfigurations. All configurations corresponding to the
entries in column $N_6$ of Table~\ref{tab:w3} can be explained in
this way. This argument can be applied recursively.

\begin{theorem}
For any integer $d$, there is a $w_3$ configuration whose smallest generating set
has size greater than $d$.
\end{theorem}

\begin{proof}
Consider the $(7m)_3$ configuration consisting of $m$ $7_3$
subconfigurations. Each of the $7_3$ subconfigurations requires 3
points for generation, so the minimum number of points in a generating
set is $3m$.
\end{proof}

\section{Practical Algorithms}\label{sect:practice}

There are two situations when fast practical algorithms for counting
configurations in Steiner triple systems are needed: if there are
many Steiner triple systems to consider, as in~\cite{CX}, or if the
order of the Steiner triple systems is large, as in~\cite{HO2}.

The main challenge in this work is that---as we are interested in
average-case performance---in a formal analysis one should know
the
distribution of all possible inputs. An
experimental approach was taken here, and algorithms
were evaluated using random Steiner triple systems. The algorithms
were produced in an exhaustive manner that will be described later
in this section. This falls within the paradigm of using algorithms to
design algorithms~\cite{Betal}. Hopefully, the computational results
will also inspire analytical studies of these algorithms.

In this section, we apply the following conventions. For a
Steiner triple system $(V,\mathcal{B})$, we let $V=\{0,1,\ldots,v-1\}$.
In $(V,\mathcal{B})$, we want to count the number of occurrences of a configuration
$(V',\mathcal{B'})$ with $|V'| = w$ points, $|\mathcal{B'}|=b$ lines, and a
minimum generating set of size $m$.

We use two auxiliary functions $B_2:\{(x,y) \in V^2 : x \ne y\} \to V$ and $B_3:V^3 \to \{0,1\}$ so that
$B_2(x,y)=z$ iff $\{x,y,z\} \in \mathcal{B}$ and
\[
B_3(x,y,z)= \left\{
\begin{array}{l}
0,\mbox{\ \ \ if\ }  \{x,y,z\} \not\in \mathcal{B},\\
1,\mbox{\ \ \ if\ }  \{x,y,z\} \in \mathcal{B}.\\
\end{array}
\right.
\]
A precomputed array of size $v^2$ can be used to evaluate both of these
functions in constant time.

The configurations $(V',\mathcal{B'})$ we focus on are the $w_3$ configurations with
$w \le 8$ and the full $n$-line configurations with $n \le 6$, that is, the nine
configurations depicted in this paper.

\subsection{Generating Sets, Up to Symmetry}\label{sect:gen}

The main idea in the algorithms to be considered is that they
loop over values for elements in a generating set of size $m$.
A configuration may have many such generating sets, but the
number of generating sets to consider can be reduced by
utilizing symmetries of the configuration, that is, its automorphism
group $\Gamma$. Since the nesting of the loops in the algorithms imply an
order on the elements of a generating set, we specifically consider
representatives from the set $\mathcal{M}$ of
transversals of the action of $\Gamma$ on the \emph{ordered} generating sets.

The automorphism group of a configuration can also be utilized
to derive conditions on its points, so that occurrences will not be counted
multiple times. We denote the orbit of an element $x$ under
the action of $G$ by $G \cdot x$ and the stabilizer by
$\operatorname{Stab}_G(x)$. Let $Z = (z_1,z_2,\ldots ,z_m)$ be
an ordered generating set that is a permutation of an element
$M \in \mathcal{M}$. For $i \in \{1, 2, \ldots, m\}$, we now compute
$O_i = \Gamma_{i-1} \cdot z_i$ and
$\Gamma_i=\operatorname{Stab}_{\Gamma_{i-1}}(z_i)$, using $\Gamma_0=\Gamma$.
Whereas $M$ gives the order of the {\bf for} loops, $Z$ shows
where the conditions given by the symmetries will be taken
into account. The value $Q := |\Gamma_m|$ gives the number of times
each configuration will be encountered.

To get an exhaustive set of algorithms, for each $M \in \mathcal{M}$,
we consider each of the $m!$ possible choices of $Z$ as permutations of $M$
and each of the $2^m$ possible choices of $E = (e_1,e_2,\ldots ,e_m)$,
$e_i \in \{\min,\max\}$, where $e_i$ tells whether the corresponding
$z_i$ should be minimum or maximum in the orbit $O_i$. As the
isomorphism $i \mapsto v-1-i$ does not change the average-case behavior,
$e_1$ can be fixed to $\min$, which leaves $2^{m-1}$ possible choices.

The following concrete example, which is split into two parts,
demonstrates how one algorithm for the Fano plane is obtained.

\vspace{5mm}
{\bf Example.} (Fano plane, Figure~\ref{sf:fano})
The lines $\mathcal{B'}$ of the Fano
plane in Figure~\ref{sf:fano} are
$\{\{\mypoint{a},\mypoint{b},\mypoint{e}\},\allowbreak{}\{\mypoint{a},\mypoint{c},\mypoint{g}\},\allowbreak{}\{\mypoint{a},\mypoint{d},\mypoint{f}\},\allowbreak{}\{\mypoint{b},\mypoint{c},\mypoint{f}\},\allowbreak{}\{\mypoint{b},\mypoint{d},\mypoint{g}\},\allowbreak{}\{\mypoint{c},\mypoint{d},\mypoint{e}\},\allowbreak{}\{\mypoint{e},\mypoint{f},\mypoint{g}\}\}$.
Given a Steiner triple system,
the goal is now to set
the seven points
$\mypoint{a},\allowbreak{}\mypoint{b},\allowbreak{}\mypoint{c},\allowbreak{}\mypoint{d},\allowbreak{}\mypoint{e},\allowbreak{}\mypoint{f},\allowbreak{}\mypoint{g}$,
which we regard as variables, such that $(\{\mypoint{a},\mypoint{b},\mypoint{c},\mypoint{d},\mypoint{e},\mypoint{f},\mypoint{g}\},\mathcal{B'})$ is a configuration of
the system.

The Fano plane has $w=7$ points, $b=7$ lines, an automorphism group $\Gamma$
of order $|\Gamma|=168$, and a minimum generating set of size $m=3$.
It has 28 minimum generating sets, which are precisely the sets of three
non-collinear points, so there are $28\cdot 3! = 168$ sets of ordered
minimum generating sets. Those 168 sets form one orbit under the action
of $\Gamma$, and we can let
$\mathcal{M}=\{(\mypoint{a},\mypoint{b},\mypoint{c})\}$.
One possible choice is $Z = (\mypoint{c},\mypoint{b},\mypoint{a})$ and
$E = (\min,\min,\min)$.

We now get $O_1=\{\mypoint{a},\mypoint{b},\mypoint{c},\mypoint{d},\mypoint{e},\mypoint{f},\mypoint{g}\}$ (the automorphism group of the Fano plane is indeed
point-transitive), $|\Gamma_1| = 24$, $O_2=\{\mypoint{a},\mypoint{b},\mypoint{d},\mypoint{e},\mypoint{f},\mypoint{g}\}$, $|\Gamma_2| = 4$,
$O_3 = \{\mypoint{a},\mypoint{d},\mypoint{e},\mypoint{g}\}$,
and $|\Gamma_3| = Q = 1$.

The constraints are then $\mypoint{c} = \min \{\mypoint{a},\mypoint{b},\mypoint{c},\mypoint{d},\mypoint{e},\mypoint{f},\mypoint{g}\}$,
$\mypoint{b} = \min \{\mypoint{a},\mypoint{b},\mypoint{d},\mypoint{e},\mypoint{f},\mypoint{g}\}$, and
$\mypoint{a} = \min \{\mypoint{a},\mypoint{d},\mypoint{e},\mypoint{g}\}$,
which can be simplified to
$\mypoint{c}<\mypoint{b}<\mypoint{a}<\mypoint{d},\mypoint{e},\mypoint{g}$ and $\mypoint{b}<\mypoint{f}$.

\vspace{5mm}
Note that if $|\Gamma_i| = 1$ for some $i < m$,
then $|\Gamma_j| = 1$ and $|O_j|=1$
for $j > i$. The minimum and maximum of a 1-element set coincide,
and we then get identical algorithms regardless of the values of
$e_j$ and $z_j$ for $j>i$. Obviously, there are then no additional
constraints on some {\bf for} loop variables. This situation
occurs especially for configurations with very small automorphism groups,
such as the Crown configuration (automorphism group order 2).
For comparison, we actually also included in our experiments all
variants of algorithms where, for $i \in \{1,2,\ldots ,m\}$,
$\Gamma_j$ and $O_j$ are not determined for $j \geq i$.
Then $Q := |\Gamma_{i-1}|$ and some
{\bf for} loop variables do not have constraints. These variants are included
later in the last column of Table~\ref{tab:genalg} but are not further
discussed here as they were not successful in the experimental evaluation.

\subsection{Algorithm for Constructing Counting Algorithms}\label{sec:AlgConCountAlg}

We are now ready to discuss our approach for exhaustively constructing
algorithms for counting $(V',\mathcal{B'})$ configurations in
$(V,\mathcal{B})$.  For clarity, we only consider the case of
generating sets of size 3, but the approach can be extended to arbitrary
sizes of generating sets with further nesting of {\bf for} loops
(\KwCont forces the next iteration of the {\bf for} loop to take place).
The missing part, $\Omega$, is explained later.

\vspace{5mm}
{\small
\begin{algorithm}[H]
$r \leftarrow 0$ \tcp*{$\mathtt{A}_1$}
\For(\tcp*[f]{$\mathtt{F}_1$}){$x_{\mathtt{F}_1} \leftarrow y_{\mathtt{F}_1}$ \KwTo $z_{\mathtt{F}_1}$}{
  \For(\tcp*[f]{$\mathtt{F}_2$}){$x_{\mathtt{F}_2} \leftarrow y_{\mathtt{F}_2}$ \KwTo $z_{\mathtt{F}_2}$}{
    \lIf(\tcp*[f]{$\mathtt{C}_{\mathtt{F}_2}$}){$\mathbf{Check}(F_2)=0$}{\KwCont}
    $x_{\mathtt{S}_1} \leftarrow B_2(x_{\mathtt{F}_1},x_{\mathtt{F}_2})$ \tcp*{$\mathtt{S}_1$}
    \lIf(\tcp*[f]{$\mathtt{C}_{\mathtt{S}_1}$}){$\mathbf{Check}(S_1)=0$}{\KwCont}
    \For(\tcp*[f]{$\mathtt{F}_3$}){$x_{\mathtt{F}_3} \leftarrow y_{\mathtt{F}_3}$ \KwTo $z_{\mathtt{F}_3}$}{
      \lIf(\tcp*[f]{$\mathtt{C}_{\mathtt{F}_3}$}){$\mathbf{Check}(F_3)=0$}{\KwCont}
      $\Omega$\;
      $r \leftarrow r + 1$ \tcp*{$\mathtt{A}_2$}
    }
  }
}
\KwRet{$r/Q$} \tcp*{$\mathtt{A}_3$}
\end{algorithm}
}

\vspace{5mm}
The comments in the right margin of the algorithm describe the
type of action taken in the respective place.

\vspace{5mm}
\begin{tabular}{lp{10cm}}
$\mathtt{A}_i$ & Actions on the accumulator $r$ for counting configurations
          ($i=1$: initializing; $i=2$: increasing; $i=3$: final division
          as each configuration is seen $Q$ times)\\
$\mathtt{C}_i$ & Checks after setting variable $x_i$\\
$\mathtt{D}_B$ & Checks regarding existence of line $B$\\
$\mathtt{F}_i$ & {\bf for} loop for variable $x_{\mathtt{F}_i}$\\
$\mathtt{S}_i$ & Fixing variable $x_{\mathtt{S}_i}$
\end{tabular}

\vspace{5mm}
The variables $x_{i}$ contain the points of the configurations; these
are either inside ($x_{\mathtt{F}_i}$, $m$ variables) or outside
($x_{\mathtt{S}_i}$, $w-m$ variables)
a generating set. After fixing any point, there is a corresponding
check $\mathtt{C_i}$, to be discussed later. After fixing a point
$x_{\mathtt{F}_k}$ in a \textbf{for} loop (and the related check),
all variables in the closure of $\{x_{\mathtt{F}_1},\ldots,x_{\mathtt{F}_k}\}$
that have not been set so far are set (rows $\mathtt{S}_i$).
In $\mathtt{D}_B$, existence of lines needed for completing the
configuration is checked; there are $b-w+m$ such lines.

Let us next elaborate on the main details.

\begin{enumerate}

\item In the \textbf{for} loops, the values of $y_{\mathtt{F}_i}$ and
$z_{\mathtt{F}_i}$ are set based on constraints on elements being
minimum or maximum in orbits, as discussed in Section~\ref{sect:gen}.
With no restrictions, we would have
$y_{\mathtt{F}_i} = 0$ and $z_{\mathtt{F}_i} = v-1$; these can
be somewhat increased and decreased, respectively, when we
have a lower bound on the number of elements that are smaller
or larger, respectively.

\item In the test \textbf{Check(X)}, we incorporate further tests of elements
being minimum or maximum in orbits. Some of the tests are included
in the \textbf{for} loops, as discussed in Item~1; any other inequalities
that can be tested are included here. Whenever a point outside the generating
set is fixed, all such tests are carried out here.
Moreover, we need to make sure that a point that is fixed differs form
the points that have been fixed earlier. The situation that a point is
not new may occur both in a \textbf{for} loop and when fixing a point in
$\mathtt{S}_i$.

\item Whenever a variable is fixed, we check existence
of the lines $B$ of the configuration that have not been involved so far
in determining new points or lines $\mathtt{D}_B$ but whose points are fixed.
Indeed, an early test makes sense since the probability
of existence of a particular line in a random STS is $1/(v-2)$.

\end{enumerate}

The missing part in the algorithm, $\Omega$, is now as follows.
For each variable $x_{\mathtt{S}_i}$ that has not yet been set, we have a code
line of type $\mathtt{S}_i$ to assign a value to $x_{\mathtt{S}_i}$.
Then we have a code line of type $\mathtt{C}_{\mathtt{S}_i}$ to test whether
all constraints are fulfilled and whether $x_{\mathtt{S}_i}$ differs from
all points that have been fixed earlier. Finally, we have one code line
of type $\mathtt{D}_B$ for each line in the configuration that consists
of fixed points but did not occur in a code line of type
$\mathtt{S}_i$ or $\mathtt{D}_B$ so far.

Note that there may be several ways of building up a configuration from
a generating set, and we consider all possible such ways in the construction
of algorithms.

We can now finish the example for Fano planes
that we started in Section~\ref{sect:gen}.

\vspace{5mm}
{\bf Example (cont).}
Based on the calculations earlier in the example, in particular the constraints
$\mypoint{c}<\mypoint{b}<\mypoint{a}<\mypoint{d},\mypoint{e},\mypoint{g}$
and
$\mypoint{b}<\mypoint{f}$, we get the following
overall structure of the algorithm, where the parts $\Omega_1$ and
$\Omega_2$ are yet to be determined:

\vspace{5mm}
\begin{algorithm}[H]
$r \leftarrow 0$\;
\For{$\mypoint{a} \leftarrow 2$ \KwTo $v-4$}{
  \For{$\mypoint{b} \leftarrow 1$ \KwTo $\mypoint{a}-1$}{
    $\Omega_1$\;
    \For{$\mypoint{c} \leftarrow 0$ \KwTo $\mypoint{b}-1$}{
      $\Omega_2$\;
      $r \leftarrow r + 1$\;
    }
  }
}
\KwRet{$r$}
\end{algorithm}

\vspace{5mm}
Next, we will discuss the remaining $w-m=4$ variables $\mypoint{d},\mypoint{e},\mypoint{f},\mypoint{g}$. The sets of new points in $\Omega_1$ and $\Omega_2$ are
$G_1=\{\mypoint{e}\}$ and $G_2=\{\mypoint{d},\mypoint{f},\mypoint{g}\}$,
respectively. We here consider one feasible ordering of the elements of
$G_1$ and $G_2$: $(\mypoint{e})$ and $(\mypoint{g},\mypoint{d},\mypoint{f})$,
respectively.

In $\Omega_1$, there is only one way of setting $\mypoint{e}$, namely
$\{\mypoint{a},\mypoint{b},\mypoint{e}\}$, and the only constraint
is $\mypoint{a}<\mypoint{e}$. This also ensures distinctness
of all points so far, so for $\Omega_1$ we get

\vspace{5mm}
\begin{algorithm}[H]
$\mypoint{e} \leftarrow B_2(\mypoint{b},\mypoint{a})$\;
\lIf{$\mypoint{e} \le \mypoint{a} $}{\KwCont}
\end{algorithm}

\vspace{5mm}
In $\Omega_2$, one choice for setting
$\mypoint{g}$, $\mypoint{d}$, and $\mypoint{f}$ is
$\{\mypoint{a},\mypoint{c},\mypoint{g}\}$,
$\{\mypoint{c},\mypoint{d},\mypoint{e}\}$, and
$\{\mypoint{e},\mypoint{f},\mypoint{g}\}$, respectively.
This choice implies that the remaining lines to check are
$\{\mypoint{a},\mypoint{d},\mypoint{f}\}$,
$\{\mypoint{b},\mypoint{c},\mypoint{f}\}$, and
$\{\mypoint{b},\mypoint{d},\mypoint{g}\}$.

For the loop variable $\mypoint{c}$ we need to ensure that
$\mypoint{c} < \mypoint{b}$.
The tests related to
$\mypoint{g}$, $\mypoint{d}$, and $\mypoint{f}$
are $\mypoint{a}<\mypoint{g}$,
$\mypoint{a}<\mypoint{d}$, and
$\mypoint{b}<\mypoint{f}$, respectively.
These are also sufficient, as it can be verified that all points
obtained in this way are necessarily distinct (for example, $\mypoint{g}
\neq \mypoint{e}$ as $\{\mypoint{a},\mypoint{c},\mypoint{g}\}$ and
$\{\mypoint{a},\mypoint{b},\mypoint{e}\}$ are distinct lines through
$\mypoint{a}$
and $\mypoint{d} \neq \mypoint{e}$ as
$\{\mypoint{c},\mypoint{d},\mypoint{e}\}$ is a line).
Altogether, for $\Omega_2$ we have

\vspace{5mm}
\begin{algorithm}[H]
$\mypoint{g} \leftarrow B_2(\mypoint{c},\mypoint{a})$\;
\lIf{$\mypoint{g} \le \mypoint{a} $}{\KwCont}
$\mypoint{d} \leftarrow B_2(\mypoint{e},\mypoint{c})$\;
\lIf{$\mypoint{d} \le \mypoint{a} $}{\KwCont}
\lIf{$ B_3(\mypoint{b},\mypoint{d},\mypoint{g})=0 $}{\KwCont}
$\mypoint{f} \leftarrow B_2(\mypoint{g},\mypoint{e})$\;
\lIf{$\mypoint{f} \le \mypoint{b} $}{\KwCont}
\lIf{$ B_3(\mypoint{a},\mypoint{d},\mypoint{f})=0 $}{\KwCont}
\lIf{$ B_3(\mypoint{b},\mypoint{c},\mypoint{f})=0 $}{\KwCont}
\end{algorithm}

\vspace{5mm}
We then get the following complete algorithm.

\vspace{5mm}
\begin{algorithm}[H]
\caption{$7_3$-configuration (Fano plane) (Figure~\ref{sf:fano})}
$r \leftarrow 0$\;
\For{$\mypoint{a} \leftarrow 2$ \KwTo $v-4$}{
  \For{$\mypoint{b} \leftarrow 1$ \KwTo $\mypoint{a}-1$}{
    $\mypoint{e} \leftarrow B_2(\mypoint{b},\mypoint{a})$\;
    \lIf{$\mypoint{e} \le \mypoint{a} $}{\KwCont}
    \For{$\mypoint{c} \leftarrow 0$ \KwTo $\mypoint{b}-1$}{
      $\mypoint{g} \leftarrow B_2(\mypoint{c},\mypoint{a})$\;
      \lIf{$\mypoint{g} \le \mypoint{a} $}{\KwCont}
      $\mypoint{d} \leftarrow B_2(\mypoint{e},\mypoint{c})$\;
      \lIf{$\mypoint{d} \le \mypoint{a} $}{\KwCont}
      \lIf{$ B_3(\mypoint{b},\mypoint{d},\mypoint{g})=0 $}{\KwCont}
      $\mypoint{f} \leftarrow B_2(\mypoint{g},\mypoint{e})$\;
      \lIf{$\mypoint{f} \le \mypoint{b} $}{\KwCont}
      \lIf{$ B_3(\mypoint{a},\mypoint{d},\mypoint{f})=0 $}{\KwCont}
      \lIf{$ B_3(\mypoint{b},\mypoint{c},\mypoint{f})=0 $}{\KwCont}
      $r \leftarrow r + 1$\;
    }
  }
}
\KwRet{$r$}
\vspace{5mm}
\end{algorithm}

\subsection{Experimental Evaluation}

We constructed all possible algorithms with the approach discussed earlier
and carried out an experimental evaluation. The authors are well aware of
the challenges involved in such work. Above all, the performance of
the algorithms depends on issues that are difficult or impossible to
control, related to compilers and microprocessors. In particular, the
programs contain many \textbf{if} statements and can be demanding
for the technique of speculative execution used commonly by
modern CPUs~\cite{K0}.

Table~\ref{tab:genalg} summarizes the main details for the generation
process of algorithms for all configurations in consideration here.
The column ``Name'' is the name of the configuration, and $b$, $w$,
$m$, $|\Gamma|$, $|\mathcal{M}_s|$, $|\mathcal{M}|$, and
$|\mathcal{M} / \Gamma|$ are the number of lines, points,
elements in a generating set of minimum size, automorphisms,
minimum generating sets, minimum ordered generating sets, and orbits
of minimum ordered generating sets under the action of the automorphism
group $\Gamma$, respectively.
The last column $A$ gives the number of distinct algorithms generated.

\begin{table}
\caption{Details for the generated algorithms}\label{tab:genalg}
\centering
\begin{tabular}{lrrrrrrrrr}
\toprule
Name & $b$ & $w$ &  $m$  & $|\Gamma|$ & $|\mathcal{M}_s|$ & $|\mathcal{M}|$  & $|\mathcal{M} / \Gamma|$  & $A$  \\
\midrule
Pasch & 4 & 6 & 3 & 24 & 16 & 96 & 4 & 296 \\
Mitre & 5 & 7 & 3 & 12 & 30 & 180 & 15 & 1272 \\
Fano--line & 6 & 7 & 3 & 24 & 28 & 168 & 7 & 2020 \\
Crown & 6 & 8 & 3 & 2 & 46 & 276 & 138 & 7348 \\
Hexagon & 6 & 8 & 3 & 12 & 48 & 288 & 24 & 2912 \\
Prism & 6 & 9 & 4 & 12 & 75 & 1800 & 150 & 60872 \\
Grid & 6 & 9 & 4 & 72 & 81 & 1944 & 27 & 34752 \\
\midrule
Fano & 7 & 7 & 3 & 168 & 28 & 168 & 1 & 828 \\
M\"{o}bius--Kantor & 8 & 8 & 3 & 48 & 48 & 288 & 6 & 9216 \\
\bottomrule
\end{tabular}
\end{table}

For each configuration, we evaluated each of the $A$ algorithms on random
Steiner triple systems
constructed by Stinson's hill-climbing algorithm~\cite{S0}.
This evaluation took part in three phases.

In the first phase, for each configuration,
we considered as many random Steiner triple systems
of order $v=93$ as could be
considered in 0.5~seconds (at least one though). In the second phase,
this was repeated for the $\max\{100,\lceil 0.01 \cdot A \rceil\}$
algorithms with the smallest average time per configuration from the first phase.
Moreover, in the second phase, 60~seconds was used for each algorithm
and the order $v=121$ was considered.
Finally, in the third phase, this was repeated for the five
algorithms with the smallest average time per configuration from the
second phase with 600~seconds time and $v=151$.

All computations were performed on 2.30-GHz Intel Xeon Gold~6140 CPUs using the
programming language \texttt{C++} compiled with the GNU Compiler Collection.

The fastest algorithm per configuration is determined as the algorithm with
smallest average time in Phase~3.
It turned out that these fastest algorithms are among the two best in Phase~2
and among the five best in Phase~1 for all configurations.
The difference in average time between the fastest and second fastest
algorithms in Phase~3, normalized to the fastest, was less than
3.2\% for all configurations that were considered.

We conclude the paper by listing the nine algorithms obtained in the aforementioned
manner as Algorithm~1 to~9;
Algorithm~1 is given in the example in Section~\ref{sec:AlgConCountAlg}.
Any scholar needing such algorithms should be able to
implement them easily in any programming language.

{
\scriptsize
\setlength{\algomargin}{0em}

\begin{minipage}[t]{0.5\textwidth}

\begin{algorithm}[H]
\caption{Pasch configuration (Figure~\ref{sf:pasch})}
$r \leftarrow 0$\;
\For{$\mypoint{a} \leftarrow 0$ \KwTo $v-6$}{
  \For{$\mypoint{b} \leftarrow \mypoint{a}+1$ \KwTo $v-2$}{
    $\mypoint{e} \leftarrow B_2(\mypoint{b},\mypoint{a})$\;
    \lIf{$\mypoint{e} \le \mypoint{a} $}{\KwCont}
    \For{$\mypoint{f} \leftarrow \max\{\mypoint{b},\mypoint{e}\}+1$ \KwTo $v-1$}{
      $\mypoint{c} \leftarrow B_2(\mypoint{f},\mypoint{a})$\;
      \lIf{$\mypoint{f} \le \mypoint{c} \lor \mypoint{c} \le \mypoint{a} $}{\KwCont}
      $\mypoint{d} \leftarrow B_2(\mypoint{f},\mypoint{e})$\;
      \lIf{$\mypoint{d} \le \mypoint{a} $}{\KwCont}
      \lIf{$ B_3(\mypoint{b},\mypoint{c},\mypoint{d})=0 $}{\KwCont}
      $r \leftarrow r + 1$\;
    }
  }
}
\KwRet{$r$}
\vspace{5mm}
\end{algorithm}

\begin{algorithm}[H]
\caption{Mitre configuration (Figure~\ref{sf:mitre})}
$r \leftarrow 0$\;
\For{$\mypoint{a} \leftarrow 0$ \KwTo $v-3$}{
  \For{$\mypoint{c} \leftarrow \mypoint{a}+1$ \KwTo $v-2$}{
    $\mypoint{f} \leftarrow B_2(\mypoint{c},\mypoint{a})$\;
    \For{$\mypoint{e} \leftarrow \max\{\mypoint{c},\mypoint{f}\}+1$ \KwTo $v-1$}{
      $\mypoint{g} \leftarrow B_2(\mypoint{f},\mypoint{e})$\;
      $\mypoint{b} \leftarrow B_2(\mypoint{g},\mypoint{a})$\;
      \lIf{$\mypoint{e} \le \mypoint{b} $}{\KwCont}
      $\mypoint{d} \leftarrow B_2(\mypoint{g},\mypoint{c})$\;
      \lIf{$\mypoint{e} \le \mypoint{d} $}{\KwCont}
      \lIf{$ B_3(\mypoint{b},\mypoint{d},\mypoint{e})=0 $}{\KwCont}
      $r \leftarrow r + 1$\;
    }
  }
}
\KwRet{$r$}
\vspace{5mm}
\end{algorithm}

\end{minipage}
\begin{minipage}[t]{0.5\textwidth}

\begin{algorithm}[H]
\caption{Fano--line configuration (Figure~\ref{sf:fanoline})}
$r \leftarrow 0$\;
\For{$\mypoint{c} \leftarrow 2$ \KwTo $v-2$}{
  \For{$\mypoint{e} \leftarrow 0$ \KwTo $\mypoint{c}-2$}{
    $\mypoint{b} \leftarrow B_2(\mypoint{c},\mypoint{e})$\;
    \For{$\mypoint{f} \leftarrow \mypoint{e}+1$ \KwTo $\mypoint{c}-1$}{
      \lIf{$\mypoint{f} = \mypoint{b} $}{\KwCont}
      $\mypoint{g} \leftarrow B_2(\mypoint{f},\mypoint{b})$\;
      \lIf{$\mypoint{g} \le \mypoint{c} $}{\KwCont}
      $\mypoint{d} \leftarrow B_2(\mypoint{e},\mypoint{g})$\;
      \lIf{$ B_3(\mypoint{c},\mypoint{d},\mypoint{f})=0 $}{\KwCont}
      $\mypoint{a} \leftarrow B_2(\mypoint{f},\mypoint{e})$\;
      \lIf{$ B_3(\mypoint{a},\mypoint{c},\mypoint{g})=0 $}{\KwCont}
      $r \leftarrow r + 1$\;
    }
  }
}
\KwRet{$r$}
\vspace{5mm}
\end{algorithm}

\begin{algorithm}[H]
\caption{Crown configuration (Figure~\ref{sf:crown})}
$r \leftarrow 0$\;
\For{$\mypoint{f} \leftarrow 0$ \KwTo $v-2$}{
  \For{$\mypoint{e} \leftarrow 0$ \KwTo $v-1$}{
    \lIf{$\mypoint{e} = \mypoint{f} $}{\KwCont}
    $\mypoint{h} \leftarrow B_2(\mypoint{f},\mypoint{e})$\;
    \For{$\mypoint{g} \leftarrow \mypoint{f}+1$ \KwTo $v-1$}{
      \lIf{$\mypoint{g} \in \{ \mypoint{e},\mypoint{h} \}  $}{\KwCont}
      $\mypoint{d} \leftarrow B_2(\mypoint{e},\mypoint{g})$\;
      $\mypoint{b} \leftarrow B_2(\mypoint{f},\mypoint{d})$\;
      $\mypoint{c} \leftarrow B_2(\mypoint{h},\mypoint{g})$\;
      \lIf{$\mypoint{c} = \mypoint{b} $}{\KwCont}
      $\mypoint{a} \leftarrow B_2(\mypoint{f},\mypoint{g})$\;
      \lIf{$ B_3(\mypoint{a},\mypoint{b},\mypoint{c})=0 $}{\KwCont}
      $r \leftarrow r + 1$\;
    }
  }
}
\KwRet{$r$}
\vspace{5mm}
\end{algorithm}

\end{minipage}
\begin{minipage}[t]{0.5\textwidth}

\begin{algorithm}[H]
\caption{Hexagon configuration (Figure~\ref{sf:hexagon})}
$r \leftarrow 0$\;
\For{$\mypoint{b} \leftarrow 0$ \KwTo $v-6$}{
  \For{$\mypoint{c} \leftarrow \mypoint{b}+2$ \KwTo $v-1$}{
    $\mypoint{a} \leftarrow B_2(\mypoint{c},\mypoint{b})$\;
    \For{$\mypoint{d} \leftarrow \mypoint{b}+1$ \KwTo $\mypoint{c}-1$}{
      \lIf{$\mypoint{d} = \mypoint{a} $}{\KwCont}
      $\mypoint{e} \leftarrow B_2(\mypoint{d},\mypoint{a})$\;
      \lIf{$\mypoint{e} \le \mypoint{b} $}{\KwCont}
      $\mypoint{h} \leftarrow B_2(\mypoint{d},\mypoint{b})$\;
      $\mypoint{g} \leftarrow B_2(\mypoint{h},\mypoint{c})$\;
      \lIf{$\mypoint{g} \le \mypoint{b} \lor \mypoint{g} = \mypoint{e} $}{\KwCont}
      $\mypoint{f} \leftarrow B_2(\mypoint{h},\mypoint{e})$\;
      \lIf{$\mypoint{f} \le \mypoint{b} $}{\KwCont}
      \lIf{$ B_3(\mypoint{a},\mypoint{f},\mypoint{g})=0 $}{\KwCont}
      $r \leftarrow r + 1$\;
    }
  }
}
\KwRet{$r$}
\vspace{5mm}
\end{algorithm}

\begin{algorithm}[H]
\caption{Prism configuration (Figure~\ref{sf:prism})}
$r \leftarrow 0$\;
\For{$\mypoint{a} \leftarrow 1$ \KwTo $v-2$}{
  \For{$\mypoint{f} \leftarrow 0$ \KwTo $\mypoint{a}-1$}{
    \For{$\mypoint{b} \leftarrow 0$ \KwTo $v-1$}{
      \lIf{$\mypoint{b} \in \{ \mypoint{a},\mypoint{f} \}  $}{\KwCont}
      $\mypoint{e} \leftarrow B_2(\mypoint{b},\mypoint{a})$\;
      \lIf{$\mypoint{e} = \mypoint{f} $}{\KwCont}
      $\mypoint{c} \leftarrow B_2(\mypoint{f},\mypoint{b})$\;
      $\mypoint{h} \leftarrow B_2(\mypoint{e},\mypoint{c})$\;
      \lIf{$\mypoint{h} \le \mypoint{a} $}{\KwCont}
      \For{$\mypoint{d} \leftarrow \mypoint{e}+1$ \KwTo $v-1$}{
        \lIf{$\mypoint{d} \in \{ \mypoint{a},\mypoint{b},\mypoint{c},\mypoint{f},\mypoint{h} \}  $}{\KwCont}
        $\mypoint{g} \leftarrow B_2(\mypoint{d},\mypoint{a})$\;
        \lIf{$\mypoint{g} \in \{ \mypoint{c},\mypoint{f},\mypoint{h} \}  $}{\KwCont}
        $\mypoint{i} \leftarrow B_2(\mypoint{h},\mypoint{d})$\;
        \lIf{$\mypoint{i} \in \{ \mypoint{b},\mypoint{f} \}  $}{\KwCont}
        \lIf{$ B_3(\mypoint{f},\mypoint{g},\mypoint{i})=0 $}{\KwCont}
        $r \leftarrow r + 1$\;
      }
    }
  }
}
\KwRet{$r$}
\vspace{5mm}
\end{algorithm}

\end{minipage}
\begin{minipage}[t]{0.5\textwidth}

\begin{algorithm}[H]
\caption{Grid configuration (Figure~\ref{sf:grid})}
$r \leftarrow 0$\;
\For{$\mypoint{a} \leftarrow 0$ \KwTo $v-9$}{
  \For{$\mypoint{b} \leftarrow \mypoint{a}+1$ \KwTo $v-3$}{
    $\mypoint{d} \leftarrow B_2(\mypoint{b},\mypoint{a})$\;
    \lIf{$\mypoint{d} \le \mypoint{b} $}{\KwCont}
    \For{$\mypoint{e} \leftarrow \mypoint{d}+1$ \KwTo $v-1$}{
      $\mypoint{g} \leftarrow B_2(\mypoint{e},\mypoint{a})$\;
      \lIf{$\mypoint{e} \le \mypoint{g} \lor \mypoint{g} \le \mypoint{a} $}{\KwCont}
      \For{$\mypoint{c} \leftarrow \mypoint{a}+1$ \KwTo $v-1$}{
        \lIf{$\mypoint{c} \in \{ \mypoint{b},\mypoint{d},\mypoint{e},\mypoint{g} \}  $}{\KwCont}
        $\mypoint{f} \leftarrow B_2(\mypoint{c},\mypoint{b})$\;
        \lIf{$\mypoint{f} \le \mypoint{a} \lor \mypoint{f} \in \{ \mypoint{e},\mypoint{g} \}  $}{\KwCont}
        $\mypoint{h} \leftarrow B_2(\mypoint{e},\mypoint{c})$\;
        \lIf{$\mypoint{h} \le \mypoint{a} \lor \mypoint{h} = \mypoint{d} $}{\KwCont}
        $\mypoint{i} \leftarrow B_2(\mypoint{g},\mypoint{f})$\;
        \lIf{$\mypoint{i} \le \mypoint{a} \lor \mypoint{i} \in \{ \mypoint{d},\mypoint{h} \}  $}{\KwCont}
        \lIf{$ B_3(\mypoint{d},\mypoint{h},\mypoint{i})=0 $}{\KwCont}
        $r \leftarrow r + 1$\;
      }
    }
  }
}
\KwRet{$r$}
\vspace{5mm}
\end{algorithm}

\begin{algorithm}[H]
\caption{$8_3$-configuration (M\"{o}bius--Kantor) (Figure~\ref{sf:mobiuskantor})}
$r \leftarrow 0$\;
\For{$\mypoint{a} \leftarrow 1$ \KwTo $v-6$}{
  \For{$\mypoint{b} \leftarrow \mypoint{a}+1$ \KwTo $v-1$}{
    $\mypoint{g} \leftarrow B_2(\mypoint{b},\mypoint{a})$\;
    \lIf{$\mypoint{g} \le \mypoint{a} $}{\KwCont}
    \For{$\mypoint{c} \leftarrow 0$ \KwTo $\mypoint{a}-1$}{
      $\mypoint{d} \leftarrow B_2(\mypoint{c},\mypoint{a})$\;
      \lIf{$\mypoint{d} \le \mypoint{a} $}{\KwCont}
      $\mypoint{h} \leftarrow B_2(\mypoint{c},\mypoint{b})$\;
      \lIf{$\mypoint{h} \le \mypoint{a} $}{\KwCont}
      $\mypoint{e} \leftarrow B_2(\mypoint{d},\mypoint{b})$\;
      \lIf{$\mypoint{e} \le \mypoint{c} $}{\KwCont}
      \lIf{$ B_3(\mypoint{e},\mypoint{g},\mypoint{h})=0 $}{\KwCont}
      $\mypoint{f} \leftarrow B_2(\mypoint{e},\mypoint{a})$\;
      \lIf{$\mypoint{f} \le \mypoint{a} $}{\KwCont}
      \lIf{$ B_3(\mypoint{c},\mypoint{f},\mypoint{g})=0 $}{\KwCont}
      \lIf{$ B_3(\mypoint{d},\mypoint{f},\mypoint{h})=0 $}{\KwCont}
      $r \leftarrow r + 1$\;
    }
  }
}
\KwRet{$r$}
\vspace{5mm}
\end{algorithm}
\end{minipage}
}


\small
\begin{thebibliography}{XX}
\bibitem{B} R. A. Beezer, Counting configurations in designs,
  \emph{J. Combin. Theory Ser. A} {\bf 96} (2001), 341--357.
\bibitem{BBP} A. Betten, G. Brinkmann, and T. Pisanski,
  Counting symmetric configurations $v_3$,
  \emph{Discrete Appl. Math.} {\bf 99} (2000), 331--338.
\bibitem{BP} J. Bokowski, and V. Pilaud, On topological and geometric
  $(19_4)$ configurations, \emph{European J. Combin.} {\bf 50} (2015), 4--17.
\bibitem{Betal} S. Brandt, J. Hirvonen, J. H. Korhonen, T. Lempi\"ainen, P. R. J.
  \"Osterg{\aa}rd, C. Purcell, J. Rybicki, J. Suomela, and P. Uzna\'{n}ski,
  LCL problems on grids, in: E. M. Schiller and A. A. Schwarzmann (Eds.),
  \emph{PODC'17, Proc. ACM Symposium on Principles of Distributed Computing
  (Washington, DC, USA, July 25--27, 2017)}, ACM Press, New York, 2017, pp. 101--110.
\bibitem{C} C. J. Colbourn, Triple systems, in: C. J. Colbourn and
  J. H. Dinitz (Eds.), Handbook of Combinatorial Designs,
  2nd ed., Chapman \& Hall/CRC, Boca Raton, 2007, pp.~58--71.
\bibitem{C0} C. J. Colbourn, The configuration polytope of $l$-line
  configurations in Steiner triple systems,
  \emph{Math. Slovaca} {\bf 59} (2009), 77--108.
\bibitem{CCS} C. J. Colbourn, M. J. Colbourn, and D. R. Stinson,
  The computational complexity of finding subdesigns in combinatorial designs,
  in: C. J. Colbourn and M. J. Colbourn (Eds.), Algorithms in Combinatorial Design
  Theory, Annals of Discrete Mathematics, 26, North-Holland, Amsterdam, 1985,
  pp.~59--65.
\bibitem{CX} C. J. Colbourn, A. D. Forbes, M. J. Grannell, T. S. Griggs, P. Kaski,
  P. R. J. \"Osterg{\aa}rd, D. A. Pike, and O. Pottonen, Properties of the Steiner
  triple systems of order 19, \emph{Electron. J. Combin.} {\bf 17} (2010), {\#}R98.
\bibitem{CR} C. J. Colbourn and A. Rosa, Triple Systems,
  Clarendon Press, Oxford, 1999.
\bibitem{DMGG} P. Danziger, E. Mendelsohn, M. J. Grannell, and T. S. Griggs,
  Five-line configurations in Steiner triple systems, \emph{Utilitas Math.}
  {\bf 49} (1996), 153--159.
\bibitem{F} A. D. Forbes, Configurations and colouring problems in block
  designs, PhD thesis, The Open University, 2006.
\bibitem{FGG} A. D. Forbes, M. J. Grannell, and T. S. Griggs, Configurations
  and trades in Steiner triple systems, \emph{Australas. J. Combin.}
  {\bf 29} (2004), 75--84.
\bibitem{GMT} M. Gezek, R. Mathon, and V. D. Tonchev, Maximal arcs, codes,
  and new links between projective planes of order 16,
  \emph{Electron. J. Combin.} {\bf 27} (2020), Paper No. 1.62.
\bibitem{GG} M. J. Grannell and T. S. Griggs, Configurations in Steiner
  triple systems, in: F. C. Holroyd, K. A. S. Quinn, C. Rowley, and B. S. Webb (Eds.),
  Combinatorial Designs and Their Applications,
  Chapman \& Hall/CRC, Boca Raton, 1999, pp.~103--126.
\bibitem{GGM} M. J. Grannell, T. S. Griggs, and E. Mendelsohn,
  A small basis for four-line configurations in Steiner triple systems,
  \emph{J. Combin. Des}. {\bf 3} (1995), pp.~51--59.
\bibitem{G} H. Gropp, Configurations, in: C. J. Colbourn and
  J. H. Dinitz (Eds.), Handbook of Combinatorial Designs,
  2nd ed., Chapman \& Hall/CRC, Boca Raton, 2007, pp.~353--355.
\bibitem{HO0} D. Heinlein and P. R. J. \"{O}sterg{\aa}rd, Dataset for Algorithms and
  Complexity for Counting Configurations in Steiner Triple Systems [Dataset]. Zenodo.
  {\tt https://doi.org/10.5281/zenodo.5519800} (September 22, 2021).
\bibitem{HO1} D. Heinlein and P. R. J. \"Osterg{\aa}rd,
  Steiner triple systems of order 21 with subsystems,
  submitted for publication.
\bibitem{HO2} D. Heinlein and P. R. J. \"Osterg{\aa}rd,
  Experimental evaluation of methods for constructing random Steiner
  triple systems, in preparation.
\bibitem{HPWY} P. Horak, N. K. C. Phillips, W. D. Wallis, and J. L. Yucas,
  Counting frequencies of configurations in Steiner triple systems,
  \emph{Ars Combin.} {\bf 46} (1997), 65--75.
\bibitem{KO} P. Kaski and P. R. J. \"Osterg{\aa}rd,
  Classification Algorithms for Codes and Designs, Springer, Berlin, 2006.
\bibitem{K0} J. Kukunas, Power and Performance: Software Analysis and Optimization,
  Elsevier, Amsterdam, 2015.
\bibitem{M0} B. D. McKay, Isomorph-free exhaustive generation,
  \emph{J. Algorithms} {\bf 26} (1998), 306--324.
\bibitem{M} G. L. Miller, On the $n^{\log n}$ isomorphism technique,
  in: \emph{Proc.\ 10th ACM symposium on the Theory of Computing
  (San Diego, Calif., 1978)}, ACM, New York, pp.~51--58.
\bibitem{S0} D. R. Stinson, Hill-climbing algorithms for the construction of combinatorial designs, \emph{Ann. Discrete Math.} {\bf 26} (1985), 321--334.
\bibitem{S} D. R. Stinson, Isomorphism testing of Steiner triple systems: canonical
  forms, \emph{Ars Combin.} {\bf 19} (1985), 213--218.
\bibitem{SW} D. R. Stinson and Y. J. Wei, Some results on quadrilaterals in
  Steiner triple systems, \emph{Discrete Math.} {\bf 105} (1992), 207--219.
\bibitem{U} E. Urland, A linear basis for the $7$-line configurations,
  unpublished.
\end{thebibliography}
\end{document}